\documentclass[12pt]{article}
\usepackage[left=3cm,top=3.0cm,right=3cm,bottom=2.6cm]{geometry}

\usepackage[ansinew]{inputenc}
\usepackage{amscd,amsfonts,amsmath,amssymb,amsthm,enumerate,epsfig,float,graphics,graphicx,pst-all,yhmath}
\newcommand{\inte}{\operatorname*{int}}
\newcommand{\aff}{\operatorname*{aff}}
\newcommand{\bd}{\operatorname*{bd}}

\newcommand{\Rn}{\mathbb{R}^{n}}

\newtheorem{lemma}{Lemma}

\newtheorem{theorem}{Theorem}
\newtheorem{corollary}{Corollary}

\newtheorem{definition}{Definition}
\newtheorem{conjecture}{Conjecture}

\newcommand{\fin}{\hfill $\Box$}

\title{Characterizations of the sphere by means of point-projections}
\author
{J. Jer\'onimo-Castro$^1$, E. Morales-Amaya$^2$, and D. J. Verdusco-Hern\'andez$^3$ \\ 
\small{$^{2,3}$Facultad de Matem\'aticas-Acapulco,}\\
\small{Universidad Aut\'onoma de Guerrero, M\'exico}\\
 \small{\texttt{$^{1}$emoralesamaya@gmail.com, $^{2}$diana.janett.h@gmail.com}}
}

\begin{document}

\maketitle

\begin{abstract}  
In this work we prove the following: let $K$ be a convex body in the Euclidean space $\mathbb{R}^n$, $n\geq 3$, contained in the interior of the unit ball of $\mathbb{R}^n$, and let $p\in \mathbb{R}^n$ be a point such that, from each point of $\mathbb{S}^{n-1}$, $K$ looks centrally symmetric and $p$ appears as the center, then $K$ is a ball.
\end{abstract}

\section{Introduction}
Consider a convex body $K$, i.e., a compact and convex set with non-empty interior in $\mathbb R^n$. As usual $\inte K$ and $\bd K$ denote the interior and boundary of $K$.  
Let $\Pi\subset \mathbb{R}^{n}$ be a hyperplane. We denote by $S_{\Pi}:\mathbb{R}^{n}\rightarrow \mathbb{R}^{n}$ the orthogonal reflection with respect to $\Pi$. We say that $K \subset \mathbb R^n$ is \emph{symmetric with respect to} $\Pi$, or that $\Pi$ is a \textit{hyperplane of symmetry} of $K$, if $S_{\Pi}(K)=K$. Let $\mathbb B^n$ and $\mathbb{S}^{n-1}$ denote the Euclidean unit ball and unit sphere of $\mathbb R^n$ with center at the origin $O$. For $u \in \mathbb{S}^{n-1}$ and $s\geq0 $, we denote by $\Pi(u,s)$ the hyperplane $\{x\in \mathbb{R}^{n} | \langle u, x\rangle = s \}$, whose unit normal vector is $u$ and its distance to the origin is equal to $s$. Moreover, we denote by $\Pi^*(u,s)$ the open half-space 
$\{x\in \mathbb{R}^{n} | \langle u, x\rangle <s \}$. For the points $x,y \in \mathbb R^n$ we  denote by $L(x,y)$ the line determined by $x$ and $y$, and by $[x,y]$ the segment with extreme points $x$ and $y$.

In order to establish our results we need to give the following definitions.

\begin{definition}
Let $K\subset \mathbb{R}^{n}$ be a convex body, $n \geq 3$, and let $x\in \mathbb{R}^{n}\setminus K$. We call the set 
$$\{x+\lambda (y-x) |  y\in K,\, \lambda\geq 0\},$$ the solid cone generated by $K$ and $x$. The boundary of this solid cone, i.e., the union of all rays starting at $x$ which do not intersect the interior of $K$, is called the cone circumscribed to $K$ with apex $x$. We denote this cone by $C_x$. 
\end{definition}

\begin{definition}
Let $C\subset\mathbb R^n$ be a convex cone with apex $x$. We say that $C$ is a symmetric cone with axis $L_x$, if there exists a line $L_x$ through $x$ such that for every $2$-dimensional plane $\Gamma$ which contains $L_x$, it holds that $L_x$ is the angle bisector of the angular region $\Gamma\cap C$. Furthermore, we say that $C$ is a right circular cone if for every hyperplane $\Pi$ orthogonal to $L_x$, with $\Pi\cap C\neq\emptyset$, the set $\Pi\cap C$ is either a point or an $(n-2)$-dimensional ball. 
\end{definition}

For $n \geq 3$ we denote by $O(n)$  the orthogonal group, i.e. the set of all the isometries of $\mathbb R^n$ that fix the origin.

Let $K\subset \mathbb R^n$ be a convex body, $n \geq 3$, and let $L$ be a line passing through the origin. We denote by $R_L :\mathbb R^n \rightarrow \mathbb R^n$ the element of $O(n)$ that acts as the identity on $L$, and sends $x$ to $-x$ on the hyperplane $L^\perp$ (the orthogonal complement of $L$). The line $L$ is said to be an axis of rotation of $K$ if the following relation holds:
\[
R_L (K) = K.
\]
Notice that if $C\subset\mathbb R^n$, $n \geq 3$, is a symmetric cone with axis $L_x$ and we choose a system of coordinates with the origin at $x$, then $L_x$ is an axis of rotation of $C$, i.e., $R_{L_x }(C) = C$.

If $K\subset \mathbb R^n,$ $n\geq 3$, is an ellipsoid, then for every $x \in \mathbb R^n \setminus K$, the cone $C_x$ is a symmetric cone. 
In fact, $C_x$ has two planes of symmetry $\Pi_1$, $\Pi_2$ such that 
$\Pi_1$, $\Pi_2$ are perpendicular and $L_x:=\Pi_1 \cap \Pi_2$, this was observed at page 24, footnote 4, of \cite{hilbert} in a statement relative to confocal system of surfaces.

Thus, the next conjecture (see \cite{odor}) seems very natural and should be true.
 
\begin{conjecture}\label{gruber}
Let $K$ be a convex body contained in the interior of $\mathbb B^n$, $n\geq 3$. If for every  $x\in \mathbb{S}^{n-1}$, $C_x$ is a symmetric cone, then $K$ is an ellipsoid.  
\end{conjecture} 

Our main result is Theorem \ref{babel}, which proves a special case of Conjecture \ref{gruber}, namely, we assume that all the axis of the cones are passing through one point. We have decided to present separately the case $n=2$ of Theorem \ref{babel}, Theorem \ref{blanco}, since it could be considered as a characterization of the circle, similar in some sense, to the main theorem in \cite{jero}.

\begin{theorem}\label{blanco}
Let $K$ be a convex body contained in the interior of $\mathbb B^2$ and let $p\in \mathbb R^n$. If for every $x\in\mathbb S^1$ we have that $L_x$ passes through $p$, then $K$ is a disc with center at $p$.
\end{theorem}
Notice that every two dimensional cone is symmetric cone, therefore $L_x$ is meaningful in dimension $n=2$.  
\begin{theorem}\label{babel}
Let $K$ be a convex body contained in the interior of $\mathbb B^n$, $n\geq 3 $,  and let $p\in \mathbb R^n$. If for every  $x\in \mathbb{S}^{n-1}$, $C_x$ is a symmetric cone and $L_x$ passes through $p$, then $K$ is a ball with center at $p$.  
\end{theorem}

As a corollary of Theorem \ref{babel} we have the following.

\begin{corollary}[\textrm{Matsuura \cite{matsura}}]\label{matsu}  
Let $K$ be a convex body contained in the interior of $\mathbb B^n$, $n\geq 3 $. If for every  $x\in \mathbb{S}^{n-1}$ we have that $C_x$ is a right circular cone, then $K$ is a ball.  
\end{corollary}

For the case when the apexes are in a hyperplane we have the following.

\begin{theorem}\label{mari}
Let $K\subset \mathbb R^n$ be a convex body, $n\geq 3 $, and let $\Pi$ be a hyperplane. If for every $x \in \Pi \setminus K$ the cone $C_x$ is a right circular cone, then $K$ is a ball.  
\end{theorem}

With a restriction in the position of the hyperplane we can say a little more.

\begin{theorem}\label{plano}
Let $K\subset \mathbb R^n$ be a convex body, $n\geq 2$, and let $\Pi$ be a hyperplane tangent to $K$. Suppose there exists a point $p$ in the interior of $K$ such that for every $x \in \Pi \setminus K$ the cone $C_x$ is symmetric and $L_x$ passes through $p$, then $K$ is a ball.  
\end{theorem}

We have obtained our results while exploring a family of problems concerning cha\-rac\-te\-ri\-za\-tion of spheres and ellipsoids in terms of geometric pro\-per\-ties of cones which circumscribe convex bodies. We considered the papers \cite{mora}, \cite{Bianchi}, \cite{matsura}, \cite{monti-mora}, and particularly Conjecture 2 in \cite{Bianchi} (which we reproduce here as Conjecture \ref{gruber}). Such conjecture was inspired by the following  characterization of the sphere due to S. Matsuura \cite{matsura}: \textit{If a convex body $K\subset \mathbb{R}^n$, $n\geq 3$, is contained in the interior of the region enclosed by a convex surface $S$, and looks spherical from each point of $S$, then $K$ is a ball}. 

The work \cite{myrosh} contains a characterization of polytopes in terms of non-central sections as well as a characterization regarding the visual recognition of polytopes. Questions related to measures of visual cones rather than shapes were also studied in \cite{Kurusa}.
We recommend seeing the Chapter 5 of \cite{gardner} for related results. 

In \cite{odor} an important evidence of the veracity of Conjecture \ref{gruber} was given, namely, it was proved there that if a convex body $K\subset \mathbb{R}^n$, $n\geq 3$, whose boundary is a surface of class $C^{4}$ with positive Gauss curvature, looks centrally symmetric from every point in the exterior of $K$, sufficiently close  to $K$, then $K$ is an ellipsoid.  

The main result in this work is Theorem \ref{babel}. Comparing with Gruber-Odor's Theorem \cite{odor}, we have reduced, substantially, the quantity of those points from where the convex body looks centrally symmetric. We assumed information of the circumscribed cones to $K$ only from the points in a sphere. However, we add the condition that there is a point $p$ that looks as the center of the body when it is observed from the points in such a sphere. As an immediate consequence of Theorem \ref{babel}, we obtained Corollary \ref{matsu} which is a special case of Matsuura's Theorem. We only considered the case $S=\mathbb{S}^{n-1}$. First we proved, under Matsuura's Theorem conditions, that all the axis of the spherical cones where the convex body is inscribed are concurrent, see Lemma \ref{erotico}. However, our proof of the general case, i.e., $n>3$, of such restricted version of Matsuura's Theorem can be given directly. In \cite{matsura}, Matsuura gives the procedure to carry out the generalization but he did not provide the complete proof. 

Theorem \ref{mari} is the natural variant of Corollary \ref{matsu}, we replaced the sphere by a hyperplane. We decided to include Theorem \ref{mari} in this work because we did not find an explicit reference about this elementary, however, interesting result.    

\section{Proofs of Theorems \ref{blanco}, \ref{babel}, and Corollary \ref{matsu}}
In order to prove Theorem \ref{blanco} we need to give some definitions and to prove a lemma which is interesting by itself. In what follows, the boundary of a given disc will be called a circle. Let $\Gamma$ be a circle in the plane and let $\Omega$ be another circle contained in the interior of $\Gamma$. For every point $x_0\in\Gamma$ we define the \emph{Poncelet-polygonal} $(x_0,x_1,x_2,x_3, \ldots)$ such that the points $x_0$, $x_1$, $x_2$, $x_3$, . . . are arranged such that the segments $[x_0,x_1]$, $[x_1,x_2]$, $[x_2,x_3]$, . . . are all tangent to $\Omega$.  The mapping $F$ such that $x_{i+1}=F(x_i)$ is called the \emph{Poncelet-mapping}. If for a positive integer number $n$ we have that $x_n=x_0$, i.e., $x_n=F^n(x_0)=x_0,$ by the well known Theorem of Poncelet (see for instance \cite{Flatto}, or \cite{Tabachnikov}) we know that for any other point $y_0\in\Gamma$ it holds that $y_n=F^n(y_0)=y_0$. If this is the case we say that $\Omega$ has the \emph{closure property} with respect to $\Gamma$. It is also known that the map $F$ has an invariant measure and hence by Denjoy's theorem (see for instance Theorem 12.3 in \cite{Flatto}) we have that $F$ is conjugate to a circle rotation. A very useful consequence of this fact is that any given circle $\Omega$, inside $\Gamma$, has either the closure property or for any point $x_0\in \Gamma$ the set $\{x_0,x_1,x_2,x_3, \ldots \}$ is a dense set in $\Gamma$ (see \cite{Flatto}).  

We say that an orientation preserving homeomorphism $F$ is conjugated to a circle rotation $R$ if there exists an homeomorphism $h:\mathbb S^1 \longrightarrow \mathbb S^1$, such that $h\circ F=R\circ h.$
 
\begin{lemma}\label{La_Catedral}
Let $\Gamma$ be a circle with center $O$ and radius $r$ and let $p$ be a point at distance $\lambda < r$ from $O$. Then, for every two numbers $r_1<r_2<r-\lambda$ there exists a number $r_3$, with $r_1<r_3 \leq r_2$, such that the circle $\Gamma_3$ with radius $r_3$ and center $p$ does not have the closure property with respect to $\Gamma$. 
\end{lemma}

\begin{figure}[H]
    \centering
    \includegraphics[width=.68\textwidth]{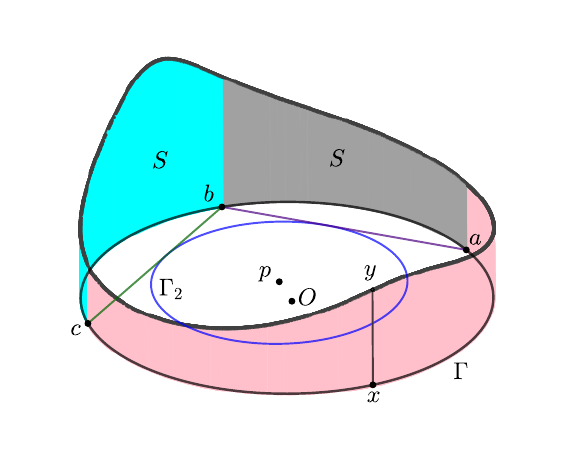}
    \caption{The area in Jacobi's surface corresponding to every chord tangent to $\Gamma_2$ is constant}
    \label{Jacobi}
\end{figure}

\emph{Proof} Consider the circle $\Gamma_2$ with center $p$ and radius $r_2$. If $\Gamma_2$ does not have the closure property with respect to $\Gamma$ then we choose $r_3=r_2$ and $\Gamma_3=\Gamma_2$. If $\Gamma_2$ has the closure property we proceed as follows: we construct the Jacobi's surface, i.e., the surface in the circular cylinder with base $\Gamma$ and such that for every point $x\in\Gamma$ its height is the reciprocal of the length of the tangent segment drawn from $x$ to $\Gamma_2$, the tangent segment is chosen in such a way that the sense of travel is counterclockwise  (see Fig. \ref{Jacobi}). A geometric interpretation of the invariant measure assigned to the Poncelet map $F$ is that for any chord of $\Gamma$, to say $[a,b]$, tangent to $\Gamma_2$, the area of the part of the Jacobi's surface over the arc $\widehat{ab}$ is equal to a constant number $S$ (see for instance \cite{Hrasko}). From this geometric interpretation the following fact is easy to see: If $A$ denotes the area of the whole Jacobi's surface, then $\Gamma_2$ has the closure property if and only if $\frac{S}{A}$ is a rational number. Indeed, the number $\frac{S}{A}$ is called the \textit{rotation number}. 

As was proved by A. O. Lopes and M. Sebastiani in \cite{Lopes}, the rotation number changes continuously for concentric circles which change their radii continuously. Hence, if we decrease continuously the radius $r_2$ and keep the center $p$, the ratio $\frac{S}{A}$ changes continuously. Thus, there exist a number $r_3$ in the open interval $(r_1,r_2)$ such that $\frac{S}{A}$ is irrational. It follows that the circle $\Gamma_3$, with center $p$ and radius $r_3$ does not have the closure property. \fin

\emph{Proof of Theorem \ref{Jacobi}.} 
First, we observe that the point $p$ is in the interior of $K$. Otherwise, we would find an $z\in \mathbb S^1$ such that $p\notin L_z$ which would contradict the hypothesis. In fact, if $p\notin \inte K$, let $W$ be a supporting line of $K$ which either \textit{separates} $p$ and $K$ or $p\in W$. Let $z$ be a point in the intersection $W\cap \mathbb S^1$. Then $p\notin L_z$ which contradicts the hypothesis.

Let $r_1$ be largest number such that the circle $\Gamma_1$ with center $p$ and radius $r_1$ is contained in $K$. Analogously, let $r_2$ be the smallest number such that the circle $\Gamma_2$ with center $p$ and radius $r_2$ encloses to $K$. If $r_1=r_2$ then $K$ is a disc. Suppose now that $r_1<r_2$. Let $x_0,x_1,x_2,x_3, \ldots$ be points in $\mathbb S^1$ arranged in counter clockwise order and such that every segment $[x_0,x_1]$, $[x_1,x_2]$, $[x_2,x_3]$, . . . is a supporting segment of $K$. By the condition that the angle bisectors $L_{x_0}$, $L_{x_1}$, $L_{x_2}$, $L_{x_3}$, . . . pass through $p$, we have that all the segments $[x_0,x_1]$, $[x_1,x_2]$, $[x_2,x_3]$, . . . are at the same distance $r$ from the point $p$. It follows that the circle with center $p$ and radius $r$ shares the tangent segments $[x_0,x_1]$, $[x_1,x_2]$, $[x_2,x_3]$, . . . with $K$. Now we apply Lemma \ref{La_Catedral} with $\Gamma=\mathbb S^1$, hence we have that there exists a number $r_3$, with $r_1<r_3\leq r_2$ such that the circle $\Gamma_3$ does not have the closure property with respect to $\mathbb S^1$. Since the distance from $p$ to the support lines of $K$ changes continuously, we have that there exists a support line $\ell$ of $K$ at distance $r_3$ from $p$. Let $x_0$ and $x_1$, in counter clockwise order, be the points where $\ell$ intersects 
$\mathbb S^1$. The Poncelet-polygonal $(x_0,x_1,x_2,x_3, \ldots)$ has its sides tangent to $\Gamma_3$ and $K$ simultaneously, however, the set $\{x_0,x_1,x_2,x_3, \ldots\}$ is a dense set in $\mathbb S^1$. It follows that every line tangent to $\Gamma_3$ is also a support line of $K$, therefore, $K$ is a disc with center $p$. \fin

\emph{Proof of Theorem \ref{babel}.} 
Notice that the point $p$ is in the interior of $K$. The proof is derived from an argument similar to the one given at the beginning of the proof of Theorem \ref{blanco}.

Let $\Gamma$ be any $2$-dimensional plane through $p$. For every point $x\in\Gamma\cap \mathbb S^{n-1}$, the line $L_x$ is the angle bisector between the two support lines of $K\cap \Gamma$ through $x$. By hypothesis, $L_x$ passes through $p$, hence we have the conditions of Theorem \ref{blanco}, and so we have that $\Gamma \cap K$ is a disc with center at $p$. Since this is true for every plane $\Gamma$ through $p$, we conclude by a theorem due to H. Busemann (see \cite{Busemann}, pp. 91-92) that $K$ is a ball with center at $p$. \fin

We first prove the following lemma and then the conclusion of the corollary follows easily.

\begin{lemma}\label{erotico}
Under the conditions of Corollary \ref{matsu}, there exists a point $p\in \inte K$ such that for every $x \in \mathbb{S}^{n-1} $ the axis $L_x$ of $C_x$ passes through $p$.
\end{lemma}

\emph{Proof.} Let $x,y\in \mathbb{S}^{n-1}$ be any two points such that $L(x,y) \cap \inte K= \emptyset$. We are going to prove that $L_x\cap L_y\neq \emptyset$. Let $\Pi_1,\Pi_2$ be two support hyperplanes of $K$ containing $L(x,y)$. It is clear that $\Pi_1,\Pi_2$ are also support hyperplanes of $C_x$ and $C_y$. Let $\Pi_{1,2}$ be the hyperplane bisecting the solid dihedral angle determined by $\Pi_1,\Pi_2$ and containing $L(x,y)$. We denote by $\Sigma$ the hyperplane $\aff\{L_x,\Pi_1\cap  \Pi_2\}$. Since for every hyperplane $\Gamma,$ with $L_x \subset \Gamma$, the equality $S_\Gamma(C_x)=C_x$ holds, it follows that $S_\Sigma(\Pi_1)$ is a support plane of $C_x$ containing $\Pi_1\cap  \Pi_2$ and different from $\Pi_1$. Thus $S_\Sigma(\Pi_1)=\Pi_2$. Hence $\Pi_{1,2}=\Sigma$. Consequently, $L_x\subset \Pi_{1,2}$. In conclusion, we have that
$$
\aff\{L_x,x,y\}=\bigcap \Pi_{1,2} 
$$
holds, where the intersection is taken over  all pairs $\Pi_1,\Pi_2$ of support hyperplanes of $K$ such that $L(x,y)\subset \Pi_1,\Pi_2$.
Interchanging $x$ by $y$, due the symmetry of this argument, we conclude that
$$
\aff\{L_y,x,y\}=\bigcap \Pi_{1,2}. 
$$
holds, where again the intersection is taken over  all pairs $\Pi_1,\Pi_2$ of support hyperplanes of $K$ such that $L(x,y)\subset \Pi_1,\Pi_2$. Therefore $\aff\{L_x,x,y\}=\aff\{L_y,x,y\}$. Since $L_x$ and $L_y$ can not be parallel, we get $L_x\cap L_y\not=\emptyset.$  

Now, we are going to prove that all lines $L_x$ pass through the same point.
On the contrary, let us assume that is not the case. Then, there exist $x, y, z \in \mathbb{S}^ {n-1}$, such that $L_x\cap L_y \cap L_ z=\emptyset$. Since every two axes intersect, the lines $L_x$, $L_y$, $L_z$ span an affine two-dimensional plane $E$. But since, for every $w \in \mathbb{S}^{n-1}$, $L_w$ intersects all $L_x$, $L_y$, $L_z$, the line $L_w$ has to be contained in $E$. However, $w\in L_w \subset E$, hence $w\in E$, for all 
$w\in \mathbb{S}^{n-1}$, leading to a contradiction. 
\fin

\emph{Proof of Corollary \ref{matsu}.} By Lemma \ref{erotico} we have that $K$ satisfies the conditions of Theorem \ref{babel}. Therefore $K$ is a ball. \fin

\section{Proof of Theorems \ref{mari} and \ref{plano}}
We first note that under the conditions of Theorem \ref{mari}, $K$ must be a strictly convex body. On the contrary, let us assume that there is a line segment $I\subset \bd K$ and we denote by $L$ the line defined by $I$. Let $\Delta$ be a supporting plane of $K$ containing $L$. We take a point $x$ in 
$\Pi\cap \Delta$ such that $x\not=L\cap \Pi$ (if $L\cap \Pi=\emptyset$, $x$ is any point in $\Pi\cap \Delta$). Let $\Gamma$ be a plane containing $L$ and not passing through $x$. Since $S(K,x)$ is a right circular cone, then $\Gamma \cap S(K,x)$ is a conic such that $I\subset \bd(\Gamma \cap S(K,x))$, however, this is absurd. Thus $K$ is a strictly convex body. 

For the following lemmas we consider $K$ to be a strictly convex body.

\begin{lemma}\label{garnachas}
Let $K\subset \mathbb R^n,$ $ n\geq 2 $, be a convex body and let $\Sigma$ be a hyperplane such that $\Sigma \cap \inte K \neq \emptyset$. If for every $(n-2)$-dimensional affine plane $\Gamma \subset \Sigma$, with $\Gamma \cap K=\emptyset$, the two support hyperplanes of $K$ containing $\Gamma$ are symmetric with respect to $\Sigma$, then $K$ is symmetric with respect to $\Sigma$. 
\end{lemma}

\emph{Proof.} Consider the reflected body $K'=S_{\Sigma}(K).$ By the hypotheses of the lemma, we have that $K$ and $K'$ share the same support hyperplanes except possibly for the points in $\Sigma\cap \bd K.$ However, $\Sigma\cap K'=\Sigma\cap K$, then we have that $K'=K$ which means that $K$ is symmetric with respect to $\Sigma$. \fin

\emph{Proof of Theorem \ref{mari}.} 
In order to prove Theorem \ref{mari} we need to define the Steiner symmetrization. Let $H\subset \Rn$ be a hyperplane. Let $C \subset \Rn$ be a nonempty compact set. According to \cite{schneider} the \textit{Steiner symmetral} of $C$ with respect to $H$ is the set $S_{H}C$ with the property that, for each line $G$ orthogonal to $H$ and meeting $C$, the set $G \cap  S_{H}C$ is a closed segment with midpoint on $H$ and length equal to that of the set $G\cap C$. The mapping $S_H : C\rightarrow S_{H}C$ is the \textit{Steiner symmetrization} with respect to $H$.

Let $H_1,...,H_k$ be hyperplanes through $o$. The mapping $S_{H_k}\circ \cdot \cdot  \cdot  \circ S_{H_1}$ is called an iterated Steiner symmetrization.
Let $\mathcal{S}(K)$ be the set of convex bodies that arise from $K$ by applying iterated Steiner symmetrizations. The Theorem10.3.2 of \cite{schneider} affirms that if $K\subset \Rn$ is a convex body, then 
$\mathcal{S}(K)$ contains a sequence that converges to a ball.

By the same argument in the beginning of the proof of Lemma \ref{erotico}, we have that there exists a point $p$ such that for every $x\in \Pi\setminus K$ the axis $L_x$ passes through $p$. 

First we prove the theorem for dimension $n=3$. We are going to prove, using Lemma \ref{garnachas}, that each plane $\Sigma$ passing through $p$ is plane of symmetry for $K$. 

Let $\Sigma$ be a $2$-dimensional plane, with $p\in \Sigma$. We denote by $L$ the intersection $\Sigma \cap \Pi$.  Let $\Gamma \subset \Sigma \setminus K$ be a line. First, we assume that 
$\Gamma$ is not parallel to $L$. Denote by $x$ the intersection $\Gamma \cap L$. If $p\in \Pi$ (see Fig. \ref{pum}), then for each line $M \subset \Pi$, $p\in M$, and for each $y\in M$, the axis $L_y$ is equal to $M$. In fact, by Lemma \ref{erotico}, for each $y\in M$, the axis $L_y$ is determined by the points $y$ and $p$. In particular, $L=L_x$ and $L\subset \Sigma$ and, consequently, we have that $\Sigma$ is a plane of symmetry for the cone $C_x$. 
In the case $p\notin \Pi$, since $p,x \in \Sigma$, then $L_x \subset \Sigma$, and again $\Sigma$ is a plane of symmetry for the cone $C_x$. Therefore, in both cases, there exist two support planes of $C_x$ and, consequently, of $K$, symmetric with respect to $\Sigma$ and containing $\Gamma$. 
Analogously, the same conclusion is obtained if we assume that 
$\Gamma$ is parallel to $L$. Thus $K$ satisfies the condition of Lemma \ref{garnachas}. Hence $\Sigma$ is a plane of symmetry for $K$.  

We take a system of coordinates with the origen $o$ at $p$. Since every plane $\Sigma$, $p\in \Sigma$, is a plane of symmetry for $K$, it follows that $K=S_{\Sigma}K$. Consequently, $\mathcal{S}(K)=\{K\}$. On the other hand, by Theorem 10.3.2 of \cite{schneider}, $\mathcal{S}(K)$ contains a sequence that converges to a ball. Thus $K$ is a ball.

\begin{figure}[H]\label{nao}
\includegraphics[width=.52\textwidth]{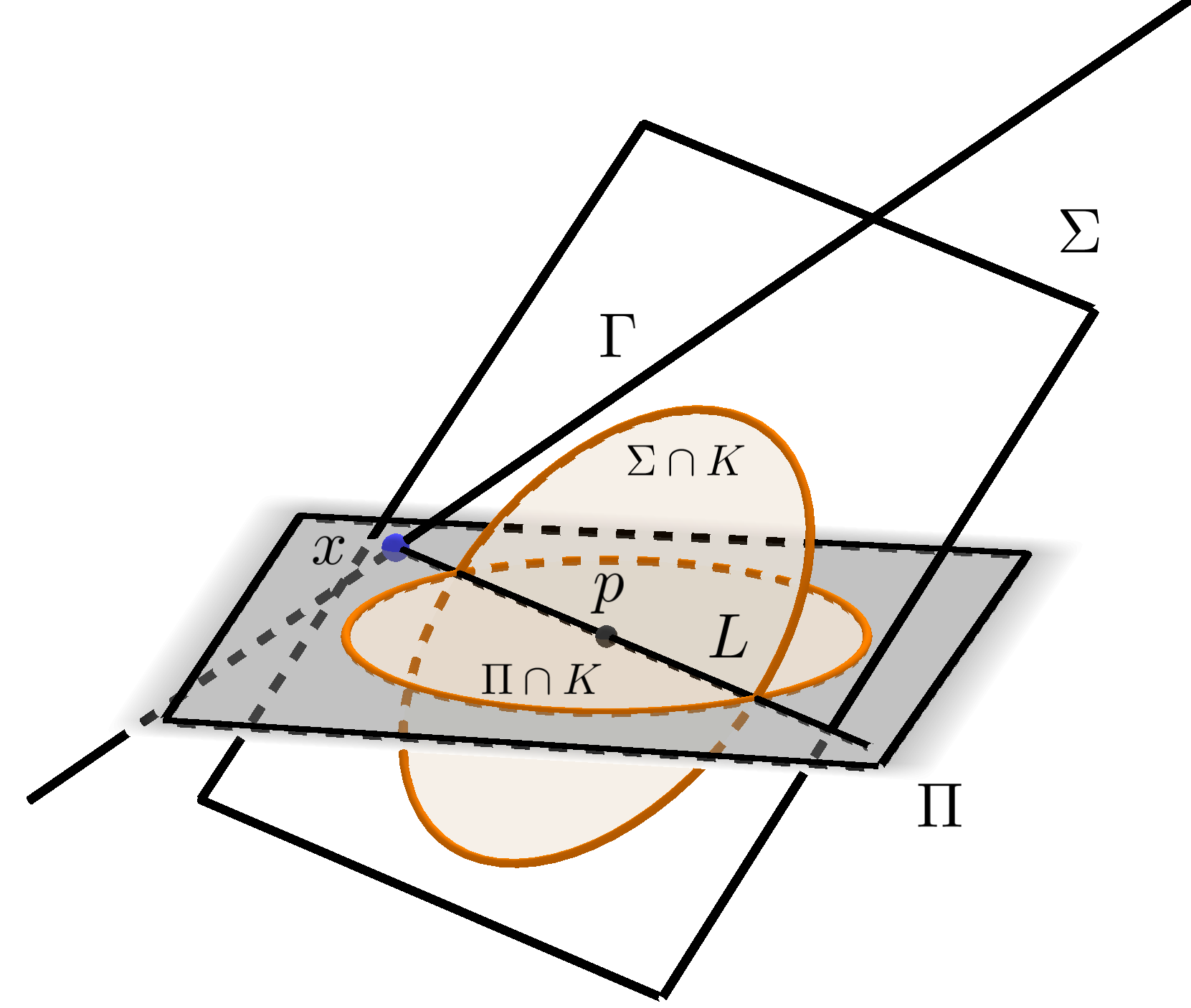}  
\centering
\caption{$\Sigma$ is a plane of symmetry of $K$.}
\label{pum}
\end{figure} 
We proceed by induction. Let $\Gamma$ be an affine $(n-1)$-dimensional plane passing through $p$. Then, for all $x\in \Gamma \cap \Pi$, the axis $L_x$ is equal to $L(x,p)$, hence $L_x\subset \Gamma$. It follows that $C_x \cap \Gamma$ is a circular cone (in dimension $n-1$) that circumscribes $\Gamma \cap K$. According to the induction hypothesis $\Gamma \cap K$ is a ball.  Hence all the $(n-1)$-dimensional sections of $K$ passing through $p$ are $(n-1)$-dimensional balls. Therefore, $K$ is a ball. \fin

\emph{Proof of Theorem \ref{plano}.} We consider first the case $n=2$. Let $x$ be any point in $\Pi\setminus K$ and let $\ell$ be the other support line of $K$ through $x$. Let $\Omega$ be the disc with center $p$ and tangent to $\Pi$ and $\ell$. For every $x$ in $\Pi$ the angle bisector of the angle circumscribed to $\Omega$ from $x$, passes through $p$. Then, $K$ and $\Omega$ share the same support lines and so they coincide, i.e., $\Omega=K$. 

Now, in dimension $n>2$ we proceed as follows: Let $z$ be a point in $\Pi\cap K$ and let $q\in \bd K$ be the point such that the segment $[z,q]$ contains $p$. Let $\Gamma$ be any $2$-dimensional plane which contains $[z,q]$ and let $\ell:=\Pi\cap \Gamma$. Let $x\in \ell$ be any point. Since the axis $L_x$ passes through $p$, by the $2$-dimensional case we have that $\Gamma\cap K$ is a $2$-dimensional disc with center at $p$. Since $\Gamma$ is any $2$-dimensional  plane which contains $[z,q]$, we conclude that $K$ is a ball with center at $p$. \fin

\section{Further comments}

Finally, we propose the following problem, which could be considered as the following natural step in way of the solution of Conjecture \ref{gruber}.

\begin{conjecture}\label{chiquita}
Let $K$ be a convex body in the interior of $\mathbb B^n$, with $n\geq 3 $, and let $L \subset \mathbb R^n$ be a line. If for every  $x\in \mathbb{S}^{n-1}$ the cone $C_x$ is a symmetric cone such that $$L_x \cap L\not=\emptyset,$$
then $K$ is an $n$-dimensional ellipsoid and for every $3$-dimensional plane $\Pi$ containing $L$, the section $\Pi \cap K$ is an ellipsoid of revolution with axis $L$.  
\end{conjecture}

\textbf{Acknowledgments}\\
We thank Sergei Tabachnikov for the helpful discussions about the proof of Theorem 1. We also thank the anonymous referees for all the comments that help to improve the paper.

\end{document}